\documentclass{amsart}
\usepackage{tikz}
\usepackage{amssymb}
\usepackage{verbatim,hyperref}
\usepackage{cleveref}
\usepackage{xy}
\xyoption{all}
\theoremstyle{plain}
\newtheorem{prop}{Proposition}[section]
\newtheorem{theorem}[prop]{Theorem}
\newtheorem{corollary}[prop]{Corollary}
\newtheorem{lemma}[prop]{Lemma}
\newtheorem{definition}[prop]{Definition}

\newcounter{lettered}
\newtheorem{letteredtheorem}[lettered]{Theorem}

\theoremstyle{definition}
\newtheorem{conventions}[prop]{Conventions}
\newtheorem{example}[prop]{Example}
\newtheorem{remark}[prop]{Remark}

\DeclareMathOperator{\avg}{{\rm avg}}

\author{A. Salch}
\address{Department of Mathematics, Wayne State University, Detroit, MI, USA}
\email{asalch@wayne.edu}
\title[Asymptotics of unlabelled bicolored graphs]{An algebraic approach to asymptotics of the number of unlabelled bicolored graphs}
\date{June 2024.}

\begin{document}

\begin{abstract}
We define and study two structures associated to permutation groups: Dirichlet characters on permutation groups, and the ``cycle form,'' a bilinear form on the group algebras of permutation groups. We use Dirichlet characters and the cycle form to find a new upper bound on the number of unlabelled bicolored graphs with $p$ red vertices and $q$ blue vertices. We use this bound to calculate the asymptotic growth rate of the number of such graphs as $p,q\rightarrow\infty$, answering a 1973 question of Harrison in the case where $q-p$ is fixed. As an application, we show that, in an asymptotic sense, ``most'' elements of the power set $P(\{ 1, \dots ,p\} \times \{ 1, \dots ,q\})$ are in free $\Sigma_p\times \Sigma_q$-orbits.
\end{abstract}
\maketitle
 
%\tableofcontents

\section{Introduction}

A ``bicolored graph'' is a graph equipped with a partition of its vertex set into two disjoint sets, one called ``red vertices'' and one called ``blue vertices,'' such that no edges connect two vertices of the same color. In an unlabelled bicolored graph, the vertices are unlabelled, but the coloring of the vertices is retained as part of the structure. It is not difficult to see that unlabelled bicolored graphs with $p$ red vertices and $q$ blue vertices correspond to orbits of the action of the product of symmetric groups $\Sigma_p\times\Sigma_q$ on the power set $P(\{ 1, \dots ,p\}\times \{1, \dots ,q\})$.

Let $p,q$ be nonnegative integers.
Let $B_u(p,q)$ denote the set of unlabelled bicolored graphs with $p$ red vertices and $q$ blue vertices.
The 1958 paper \cite{MR0103834} of Harary (see \cite[pg. 7, ``Product Group Enumeration Theorem'']{MR0249329} for a more direct statement) uses P\'{o}lya enumeration to show that the cardinality of $B_u(p,q)$ is given by the formula 
\begin{align} 
\label{harary formula}
 \left|B_u(p,q)\right| &= \frac{1}{p!q!} \sum_{\alpha\in \Sigma_p}\sum_{\beta\in \Sigma_q} \prod_{r,s} 2^{\gcd(r,s)\cdot c_r(\alpha)\cdot c_s(\beta)},
\end{align}
where $c_r(\alpha)$ is the number of $r$-cycles in the permutation $\alpha$. The same formula was obtained, evidently independently, by M. A. Harrison in \cite{MR0357155}. In the same paper, Harrison asked about the {\em asymptotic} properties of the function $\left|B_u(p,q)\right|$. Little seems to have been done to pursue Harrison's question for the next 45 years, until the 2018 paper \cite{MR3901841}, which proves bounds\footnote{The paper \cite{MR3901841} refers to ``bipartite'' graphs throughout, but this appears to be idiosyncratic, and ``bicolored'' seems to be meant instead. Here is a terminological note to explain the situation. A ``bipartite graph,'' also called a ``bicolorable graph,'' is a graph that {\em admits} a bicoloring, but is not equipped with a {\em choice} of bicoloring. Counting unlabelled bicolored graphs, as done in \cite{MR0103834}, is a straightforward case of P\'{o}lya enumeration. It takes more work to count unlabelled bipartite graphs, as in \cite{MR0542935} and \cite{MR0165512}, or connected unlabelled bipartite graphs, as in \cite{MR3244806}. 

The present paper is entirely about unlabelled bicolored graphs, so the {\em counting} problem was settled straightforwardly a long time ago. The question of asymptotic growth of the resulting count has remained open, however: it is partially addressed in \cite{MR3901841}, and in case $q-p$ remains fixed as $p,q\rightarrow \infty$, it is more completely addressed in this paper.} 
\begin{equation}\label{atmaca-oruc ineqs 2} \frac{1}{q!}\binom{p+2^q-1}{p} \leq \left| B_u(p,q)\right| \leq \frac{2}{q!}\binom{p+2^q-1}{p}.\end{equation}

In \cref{dir chars section} of this paper we develop a rudimentary theory of ``Dirichlet characters on permutation groups.'' In particular, for each complex number $z$ and each symmetric group $G$, we construct a unique ``cyclic Dirichlet character'' $\chi_z: G \rightarrow \mathbb{C}$ which sends every transposition in $G$ to the complex number $z$. There is a certain product on cyclic Dirichlet characters, written $((-,-))$, which we define in Definition \ref{def of twisted product of dir chars}. The product $((\chi,\chi^{\prime}))$ of two cyclic Dirichlet characters is a complex number.  
We use these Dirichlet characters to formulate and prove an upper bound for $\left| B_u(p,q)\right|$ which is stronger than the Atmaca--Oru\c{c} upper bound \cite{MR3901841}, namely,
\begin{letteredtheorem}[Theorem \ref{main thm}]\label{letteredthm a}
\begin{align*}
 \left| B_u(p,q)\right|
  &\leq 2^{pq/2} ((\chi_{1/2}, \chi_{2^{q/2}})).
\end{align*}
\end{letteredtheorem}
During our development of the theory of Dirichlet characters on permutation groups, we prove in Proposition \ref{dir char product bound} a bound on the asymptotic growth of the product $((\chi_{1/2}, \chi_{2^{q/2}}))$. As a consequence of this asymptotic bound and Theorem \ref{letteredthm a}, we get
\begin{letteredtheorem}[Corollary \ref{growth cor}] \label{letteredthm b}
For any integer $k$, the number $\left| B_u(p,p+k)\right|$ grows asymptotically no faster than $\frac{2^{p(p+k)}}{p!(p+k)!}$. That is, $\lim_{p\rightarrow \infty} \frac{\left| B_u(p,p+k)\right|}{2^{p(p+k)}/(p!(p+k)!)} \leq 1$.
\end{letteredtheorem}
In Remark \ref{remark on atmaca-oruc} we point out that our {\em upper} bound $\frac{2^{p(p+k)}}{p!(p+k)!}$ for $\left| B_u(p,p+k)\right|$ has the same asymptotic growth as the {\em lower} bound for $\left| B_u(p,p+k)\right|$ obtained by Atmaca--Oru\c{c} in \eqref{atmaca-oruc ineqs 2}. Consequently we obtain an answer to Harrison's question about the asymptotic properties of the function $\left|B_u(p,q)\right|$, at least when $q = p+k$ for fixed $k$:  as $p\rightarrow\infty$, the function $\left|B_u(p,p+k)\right|$ is in the same asymptotic growth class as $\frac{2^{p(p+k)}}{p!(p+k)!}$. 

Finally, we offer an application for our asymptotic bound. The product of symmetric groups $\Sigma_p \times \Sigma_q$ acts on the power set $P(\{ 1, \dots ,p\}\times\{ 1, \dots ,q\})$ in the evident way. One can ask  what proportion of the elements of $P(\{ 1, \dots ,p\}\times\{ 1, \dots ,q\})$ live in a {\em free} $\Sigma_p\times\Sigma_q$-orbit. Write $f(p,q)$ for this proportion.

It is not difficult to see that, for fixed $q$, the limit $\lim_{p\rightarrow\infty}f(p,q)$ is not equal to $1$. For example, $f(p,1)$ is zero for all $p>2$, so $\lim_{p\rightarrow\infty} f(p,1) = 0$. However, if we fix an integer $k$ and let $q = p+k$, then in the limit as $p\rightarrow\infty$, most elements of $P(\{1, \dots ,p\}\times \{1, \dots ,q\})$ {\em do} live in free $\Sigma_p\times\Sigma_{q}$-orbits:
\begin{letteredtheorem}[Theorem \ref{main cor}]\label{letteredthm c}
$\lim_{p\rightarrow\infty} f(p,p+k) = 1$. 
\end{letteredtheorem}

In the literature, we were not able to locate any analysis of the proportion of orbits in $P(\{1, \dots ,p\}\times \{1, \dots ,q\})$ which live in free $\Sigma_p\times\Sigma_q$-orbits. Nevertheless, other proofs of Theorem \ref{letteredthm c} are possible: see for example the nice argument \cite{471668} posted by the pseudonymous Fedja when the author asked on MathOverflow whether Theorem \ref{letteredthm c} was already known to combinatorialists. Our proof has the virtue of being an immediate corollary of our asymptotic analysis of $\left| B_u(p,p+k)\right|$ from Theorem \ref{letteredthm b}, which in turn is largely a consequence of structural results we prove about behavior of Dirichlet characters on permutations and also the behavior of a certain bilinear form $\langle -,-\rangle: \mathbb{Z}[G]\otimes_{\mathbb{Z}}\mathbb{Z}[H]\rightarrow\mathbb{Z}$ on the group algebras of permutation groups, the ``cycle form,'' which we construct and study in \cref{cycle form section}. (Of course, if all that is desired is Theorem \ref{letteredthm c}, then it is not necessary to talk about the structural results, or to bound $\left| B_u(p,p+k)\right|$, etc.; a direct proof such as Fedja's is much shorter than doing all that.)

We think these structural results can be of some interest in their own right. 
In principle, this paper could have been shorter by eliminating the structural study of Dirichlet characters and the cycle form, proving the same estimates on $\left| B_u(p,q)\right|$ by essentially the same arguments but without using the language and general properties developed in that structural study. But the resulting arguments are harder to follow and the ideas are less clear. Perhaps Dirichlet characters on permutation groups or the cycle form can also be useful in other problems in algebraic combinatorics. Hence we think it is worthwhile to present the arguments in this algebraic way.

\begin{conventions}\label{conventions}\leavevmode
\begin{itemize}
\item Given a permutation $\sigma$ of some finite set, we will write $c(\sigma)$ for the number of cycles of $\sigma$, {\em including} singleton cycles.
\item Let $i$ be a positive integer. When a symmetric group $\Sigma_p$ is understood from context, we will write $\gamma_i$ to mean an arbitrary $i$-cycle in $\Sigma_p$.
\item We write $x^{\overline{k}}$ to mean the rising factorial function, i.e., $x^{\overline{k}} = x(x+1)\dots (x+k-1)$.
\end{itemize}
\end{conventions}

\section{Dirichlet characters on permutation groups.}
\label{dir chars section}

Here is a classical definition from number theory. Given a positive integer $m$ and a function $\chi:\mathbb{Z}/p\mathbb{Z}\rightarrow \mathbb{C}$ which vanishes on the residue classes which are not coprime to $m$, we say that $f$ is a {\em Dirichlet character of modulus $m$} if the following conditions\footnote{The clause ``if $\gcd(j,k)=1$'' in the second condition is deliberately redundant. If a function $\chi: \mathbb{Z}/m\mathbb{Z}\rightarrow\mathbb{C}$ vanishes on the residue classes prime to $m$ and satisfies the two stated conditions, then in fact $\chi(jk) = \chi(j)\chi(k)$. This is for the elementary reason that, for any two integers $j,k$ coprime to $m$, there is some multiple $\alpha m$ of $m$ such that $j$ and $k+\alpha m$ are relatively prime, hence $\chi(jk) = \chi(j)\chi(k+\alpha m) = \chi(j)\chi(k)$. 
%To prove that claim: choose some integer n such that j is 
% rel. prime to k+n. Use the Euclidean algorithm to find integers
% c,d such that 1 = cj + dm. Then let \alpha = dn. You have:
% gcd(j,k+ \alpha m) = gcd(j, k+n) = 1.

The redundant clause ``if $\gcd(j,k)=1$'' is usually omitted when defining a Dirichlet character. We include the clause in our definition because it makes the comparison to Definition \ref{def of dir char of perm grp} more natural.} are satisfied:
\begin{itemize}
\item $\chi(1) = 1$, and
\item $\chi(jk) = \chi(j)\chi(k)$ if $\gcd(j,k)=1$.
\end{itemize}
Of course Dirichlet characters of modulus $m$ are equivalent to group homomorphisms $\mathbb{Z}/m\mathbb{Z}^{\times}\rightarrow\mathbb{C}^{\times}$.

Here is a combinatorial analogue of Dirichlet characters:
\begin{definition}\label{def of dir char of perm grp}
Given a permutation group $G$, a {\em Dirichlet character on $G$} is a class\footnote{Recall that a {\em class function} is a function on a group $G$ which is invariant under conjugacy in $G$. Since $\mathbb{Z}/m\mathbb{Z}$ is commutative, this condition is moot in the classical setting of Dirichlet characters in number theory.} function $\chi: G\rightarrow\mathbb{C}$ such that
\begin{itemize}
\item $\chi(1) = 1$, and
\item $\chi(\sigma_1\sigma_2) = \chi(\sigma_1)\chi(\sigma_2)$ if the permutations $\sigma_1$ and $\sigma_2$ are disjoint.
\end{itemize}
\end{definition}
%Dirichlet characters on a permutation group $G$ form an abelian group, in fact a $\mathbb{Q}$-vector space, under pointwise multiplication. (MAYBE NO INVERSES DUE TO ZERO VALUES OF CHARACTERS?) We write $\Dir(G)$ for this $\mathbb{Q}$-vector space. 

It is reasonable to build up a theory of Dirichlet characters on a general permutation group $G$, but in this paper all our applications are limited to the case where $G$ is a symmetric group. From now on, we will restrict ourselves to that level of generality.

Suppose that $p$ is a positive integer, and suppose that $\chi: \Sigma_p\rightarrow\mathbb{C}$ is a Dirichlet character. Then the value of $\chi$ on an element $\sigma\in\Sigma_p$ depends entirely by the number of cycles of each length in $\chi$. That is, $\chi$ is determined by a sequence of complex numbers $z_1,z_2, \dots ,z_p$, where $z_i$ is the value of $\chi$ on an $i$-cycle. Furthermore we must have $z_1 = 1$, since $\chi(1) = 1$.

In the simplest case, these numbers are all powers of $z_2$:
\begin{definition}
We say that a Dirichlet character $\chi:\Sigma_p\rightarrow\mathbb{C}$ is {\em cyclic} if, for every $i$-cycle $\gamma_i$ and every transposition $\tau$, we have $\chi(\gamma_i) = \chi(\tau)^{i-1}$.
\end{definition}

Recall from Conventions \ref{conventions} that we use the notation $c(\sigma)$ for the number of cycles in a permutation $\sigma$, including singleton cycles.
\begin{prop}\label{characterization of cyclic chars}
Let $p$ be a positive integer. A Dirichlet character $\chi:\Sigma_p\rightarrow\mathbb{C}$ is cyclic if and only if both of the following conditions are satisfied:
\begin{enumerate} 
\item for every pair of permutations $\sigma,\sigma^{\prime}\in \Sigma_p$ such that $c(\sigma) = c(\sigma^{\prime})$, we have $\chi(\sigma) = \chi(\sigma^{\prime})$, and
\item for every transposition $\tau$ and every cycle $\gamma$ such that some element of $\{1, \dots,p\}$ fixed by $\gamma$ is not fixed by $\tau$, we have $\chi(\tau\gamma) = \chi(\tau)\chi(\gamma)$. 
\end{enumerate}
\end{prop}
The second condition in Proposition \ref{characterization of cyclic chars} is difficult to parse, but is much clearer from examples. The condition asserts that, for example, $\chi((123)(35)) = \chi((123))\chi((35))$, since the element $5\in \{ 1, \dots 5\}$ is fixed by $(123)$ but not fixed by $(35)$. Similarly, the second condition asserts that $\chi((123)(45)) = \chi((123))\chi((45))$. The second condition does {\em not} assert that $\chi((123)(13)) = \chi((123))\chi((13))$.
\begin{proof}[Proof of Proposition \ref{characterization of cyclic chars}]
Suppose $\chi$ is cyclic. Write $\sigma\in \Sigma_p$ as a product $\sigma = \tau_1 \dots \tau_{c(\sigma)}$ of disjoint cycles, including singletons. Write $\ell(i)$ for the length of the cycle $\tau_i$. Then $\chi(\sigma) = \prod_{i=1}^{c(\sigma)} \chi(\tau_i) = \prod_{i=1}^{c(\sigma)} \chi((12))^{\ell(i) - 1} = \chi((12))^{p-c(\sigma)}$. Hence the value of $\chi$ on a permutation $\sigma$ depends only on $c(\sigma)$ and on $\chi((12))$, so the first condition is satisfied.

Given a transposition $\tau$ and a cycle $\gamma$ satisfying the constraints of the second condition, there are two possibilities: either $\tau$ and $\gamma$ are disjoint, or they are not. We handle the cases separately:
\begin{description}
\item[If $\tau,\gamma$ are disjoint] then $\chi(\tau\gamma) = \chi(\tau)\chi(\gamma)$ by the definition of a Dirichlet character on a permutation group.
\item[If $\tau,\gamma$ are not disjoint] then without loss of generality we may assume that $\gamma = (1...n)$ and $\tau = (1,n+1)$ for some $n$. We then have
\begin{align*}
 \chi(\gamma\tau) 
  &= \chi((1,...,n+1)) \\
  &= \chi(\tau)^n \\
  &= \chi(\tau)^{n-1}\chi(\tau) \\
  &= \chi((1...n))\chi((1,n+1)),
\end{align*}
as desired.
\end{description}

For the converse: suppose the two conditions in the statement of the proposition are satisfied. By the first condition, to verify that $\chi$ is cyclic, it suffices to verify that $\chi((1...n)) = \chi((12))^{n-1}$ for each $n= 1, \dots ,p$. 
This follows from a straightforward induction:
\begin{itemize}
\item $\chi((123)) = \chi((12))\chi((23)) = \chi((12))^2$,
\item $\chi((1234)) = \chi((123))\chi((34)) = \chi((12))^3$,
\item $\chi((12345)) = \chi((1234))\chi((45)) = \chi((12))^4$,
\item and so on, with the right-hand equalities provided by the second condition in the statement of the proposition.
\end{itemize}
\end{proof}

Proposition \ref{characterization of cyclic chars} offers a way for us to recognize whether a given Dirichlet character is cyclic. But if one simply wants a list of all cyclic Dirichlet characters, this is easier: to specify a cyclic Dirichlet character $\chi$ on $\Sigma_p$, one simply gives the value of $\chi$ on any transposition in $\Sigma_p$. This value can be any complex number. Hence there is precisely one cyclic Dirichlet character on $\Sigma_p$ for each complex number. 
\begin{definition}
Given a positive integer $p$ and an element $z\in \mathbb{C}^{\times}$, we will write $\chi_z$ for the unique cyclic Dirichlet character $\Sigma_p\rightarrow\mathbb{C}$ which sends a transposition to $z$.
\end{definition}

The signless Stirling number of the first kind, $c(p,k)$, counts the number of elements of $\Sigma_p$ which have precisely $k$ cycles, including singleton cycles in the count. Consequently the average value $\avg(\chi)$ of a cyclic Dirichlet character $\chi$ on the symmetric group $\Sigma_p$ is
given by the formula
\begin{align*}
 \frac{1}{p!} \sum_{\sigma\in \Sigma_p} \chi(\sigma)
  &= \frac{1}{p!} \sum_{k=1}^p c(p,k) \sum_{c(\sigma)=k} \chi(\sigma) \\
\nonumber  &= \frac{1}{p!} \sum_{k=1}^p c(p,k) \chi(\tau)^{p-k} 
%\nonumber  &= \frac{\chi(\tau)^p(\chi(\tau)^{-1})^{\overline{p}}}{p!},
\end{align*}
where $\tau$ is any transposition in $\Sigma_p$.
As a simple consequence, 
\begin{align}\label{averaging formula 1}
 \frac{\avg(\chi)}{\chi(\tau)^p} 
  &= \frac{1}{p!} \sum_{k=1}^p c(p,k) \chi(\tau)^{-k} \\
\nonumber &= \frac{(\chi(\tau)^{-1})^{\overline{p}}}{p!},
\end{align}
where $z^{\overline{p}}$ is the rising factorial $z(z+1)(z+2)\dots (z+p-1)$. 

We shall have need of a ``twisted''
\begin{comment}
\footnote{It is clear that a reasonable ``untwisted'' analogue of $((\chi,\chi^{\prime}))$ is
\begin{align}\label{untwisted form}
\frac{1}{p!\cdot q!} \sum_{\alpha\in\Sigma_p} \sum_{\beta\in\Sigma_q} \chi(\alpha) \chi^{\prime}(\beta),
\end{align}
at least up to some constant depending only on $p$ and $q$.
For cyclic $\chi$ and $\chi^{\prime}$, \eqref{untwisted form} is equal to $\left( \frac{1}{p!} \sum_{\alpha\in\Sigma_p}\chi(\alpha) \right)\left( \frac{1}{q!}\sum_{\beta\in\Sigma_q}\chi^{\prime}(\beta)\right)$, i.e., the product of the average value of $\chi$ with the average value of $\chi^{\prime}$. The formula \eqref{untwisted form} is more natural than the ``twisted'' product $(\chi,\chi^{\prime})$, and \eqref{untwisted form} makes sense even for non-cyclic Dirichlet characters $\chi,\chi^{\prime}$. It even yields a bilinear form $\Dir(\Sigma_p)\otimes_{\mathbb{Q}} \Dir(\Sigma_q)\rightarrow\mathbb{C}$. 

The reader is likely asking: why not study \eqref{untwisted form}, rather than the more convoluted product $((\chi,\chi^{\prime}))$? The reason is simply that $((\chi,\chi^{\prime}))$ shows up naturally in our bound for the number of bicolored graphs, in Theorem \ref{main thm}, while \eqref{untwisted form} does not.}
\end{comment}
two-character analogue of the formula \eqref{averaging formula 1}:
\begin{definition}\label{def of twisted product of dir chars}
Fix positive integers $p$ and $q$. Given cyclic Dirichlet characters $\chi: \Sigma_p\rightarrow\mathbb{C}$ and $\chi^{\prime}: \Sigma_q\rightarrow \mathbb{C}$, we define $((\chi,\chi^{\prime}))$ to be the complex number 
\begin{align*}
 ((\chi,\chi^{\prime}))
  &= \frac{1}{p! q!} \sum_{k=1}^p \sum_{\ell=1}^q c(p,k)c(q,\ell) \chi(\tau)^{-k\ell} \chi^{\prime}(\tau^{\prime})^{-k},\\
  &= \frac{1}{p! q!} \sum_{\alpha\in \Sigma_p} \sum_{\beta\in \Sigma_q} \chi(\tau)^{-c(\alpha) c(\beta)} \chi^{\prime}(\tau^{\prime})^{-c(\alpha)},
\end{align*} 
where $\tau$ is any transposition in $\Sigma_p$, and $\tau^{\prime}$ is any transposition in $\Sigma_q$.
\end{definition}

Here is an extremely elementary lemma:
\begin{lemma}\label{am-gm lemma}
Let $a,b$ be positive integers. Then $a^{\overline{b}} \leq (a+\frac{b-1}{2})^b$.
\end{lemma}
\begin{proof}
The $b$th root of the rising factorial $a^{\overline{b}}$ is the geometric mean of the integers $a,\ a+1,\ \dots ,\ a+b-1$, hence is bounded above by the arithmetic mean of those numbers, which is $a+\frac{b-1}{2}$.
\end{proof}

The following proposition establishes that, as $p\rightarrow\infty$, the twisted product of cyclic Dirichlet characters $((\chi_{1/2},\chi_{2^{(p+k)/2}}))$ grows asymptotically no faster than $2^{p(p+k)/2}/(p!(p+k)!)$. This asymptotic estimate is used in Corollary \ref{growth cor} and Theorem \ref{main cor}.
\begin{prop}\label{dir char product bound}
Let $p$ be a positive integer and let $k$ be a nonnegative integer. Then we have an inequality:
\begin{align*}
\lim_{p\rightarrow \infty} \frac{((\chi_{1/2},\chi_{2^{(p+k)/2}}))\cdot p!(p+k)!}{2^{p(p+k)/2}} &\leq 1.
\end{align*}
\end{prop}
\begin{proof}
From elementary algebraic manipulation:
\begin{align}
\nonumber \lim_{p\rightarrow \infty} \frac{((\chi_{1/2},\chi_{2^{(p+k)/2}}))\cdot p!(p+k)!}{2^{p(p+k)/2}}
%  &= \lim_{p\rightarrow \infty} \frac{\sum_{\alpha\in\Sigma_p} \sum_{\beta\in\Sigma_{p+k}} 2^{c(\alpha)c(\beta)}2^{-c(\alpha)\cdot (p+k)/2}}{2^{p(p+k)/2}} \\
%\nonumber  &= \lim_{p\rightarrow \infty} \frac{\sum_{i=0}^p\sum_{j=0}^{p+k} c(p,i) c(p+k,j) 2^{ij}(2^{-(p+k)/2})^i}{2^{p(p+k)/2}} \\
\nonumber  &= \lim_{p\rightarrow \infty} \frac{\sum_{i=0}^p\sum_{j=0}^{p+k} c(p,i) c(p+k,j) (2^{j-(p+k)/2})^i}{2^{p(p+k)/2}} \\
\nonumber  &= \lim_{p\rightarrow \infty} \frac{\sum_{j=0}^{p+k} c(p+k,j) (2^{j-(p+k)/2})^{\overline{p}}}{2^{p(p+k)/2}} \\
\label{ineq 001}  &\leq \lim_{p\rightarrow \infty} \frac{\sum_{j=0}^{p+k} c(p+k,j) (2^{j-(p+k)/2} + \frac{p-1}{2})^p}{2^{p(p+k)/2}} \\
\nonumber  &= \lim_{p\rightarrow \infty} \frac{\sum_{j=0}^{p+k} c(p+k,j) (2^{j} + \frac{p-1}{2}2^{(p+k)/2})^p}{2^{p(p+k)}} \\
\nonumber  &= \lim_{p\rightarrow \infty} \frac{\sum_{j=0}^{p+k} c(p+k,j) \sum_{h=0}^p \binom{p}{h} (2^h)^j\left( \frac{p-1}{2} 2^{(p+k)/2}\right)^{p-h}}{2^{p(p+k)}} \\
\nonumber  &= \lim_{p\rightarrow \infty} \frac{\sum_{h=0}^p \binom{p}{h} \left( \frac{p-1}{2} 2^{(p+k)/2}\right)^{p-h}(2^h)^{\overline{p+k}}}{2^{p(p+k)}} \\
\nonumber  &= \lim_{p\rightarrow \infty}\left( \frac{(2^p)^{\overline{p+k}}}{(2^p)^{p+k}} + 
\frac{\sum_{h=0}^{p-1} \binom{p}{h} \left( \frac{p-1}{2} 2^{(p+k)/2}\right)^{p-h}(2^h)^{\overline{p+k}}}{2^{p(p+k)}}\right) \\
\label{eq 004}  &= \lim_{p\rightarrow \infty} \frac{(2^p)^{\overline{p+k}}}{(2^p)^{p+k}} + \lim_{p\rightarrow \infty}\sum_{h=0}^{\infty} a_{h,p},
\end{align}
where inequality \eqref{ineq 001} is due to Lemma \ref{am-gm lemma}, and where $a_{h,p}$ is the real number defined as follows:
\begin{align*}
 a_{h,p} &= \left\{ \begin{array}{ll}
  0 &\mbox{\ if\ } h<0 \\
  \binom{p}{h} \left( \frac{p-1}{2} 2^{(p+k)/2}\right)^{p-h}(2^h)^{\overline{p+k}}/2^{p(p+k)} &\mbox{\ if\ } 0\leq h< p\\
  0 &\mbox{\ if\ } h\geq p .\end{array}\right.
\end{align*}
It is easy to see that $\lim_{p\rightarrow \infty} a_{h,p} = 0$ for each integer $h$. Hence, if we can exchange the limit with the sum in \eqref{eq 004}, then \eqref{eq 004} will be equal to $\lim_{p\rightarrow \infty} \frac{(2^p)^{\overline{p+k}}}{(2^p)^{p+k}}$, i.e., equal to $1$, as desired.

It is a straightforward matter of dominated convergence to exchange the limit with the sum in \eqref{eq 004}, as follows. There exists some positive integer $H_k$ such that, for all $h\geq H_k$, the largest value of the sequence $a_{h,1},a_{h,2},a_{h,3}, \dots$ is its first nonzero term, i.e., $a_{h,h+1} = \frac{(h+1)h}{2} 2^{(h+1+k)(-h-1/2)}2^{\overline{h+1+k}}$. (Explicit calculation shows that $H_0 = 12$ suffices, while $H_1 = 10$ and $H_2 = 7$ and $H_k = 1$ for all $k\geq 3$ also suffice.) Since $\lim_{h\rightarrow \infty} a_{h+1,h+2}/a_{h,h+1} = 0$ by straightforward calculation, the sum $\sum_{h\geq H_k} a_{h,h+1}$ converges. Hence by Tannery's theorem, we get the equality \eqref{dct eq 1} in the chain of equalities
\begin{align}
\nonumber 0 
  &= \sum_{h=0}^{\infty}\lim_{p\rightarrow \infty} a_{h,p} \\
\nonumber  &= \lim_{p\rightarrow \infty}\sum_{h=0}^{H_k-1} a_{h,p} + \sum_{h=H_k}^{\infty} \lim_{p \rightarrow \infty} a_{h,p} \\
\label{dct eq 1} &= \lim_{p\rightarrow \infty}\sum_{h=0}^{H_k-1} a_{h,p} + \lim_{p\rightarrow \infty}\sum_{h=H_k}^{\infty} a_{h,p} \\
\nonumber  &= \lim_{p\rightarrow \infty}\sum_{h=0}^{\infty} a_{h,p},
\end{align}
as desired.
%\label{ineq 002}  &\leq \lim_{p\rightarrow \infty} \frac{\sum_{h=0}^p \binom{p}{h} \left( \frac{p-1}{2} 2^{(p+k)/2}\right)^{p-h}(2^h + \frac{p+k-1}{2})^{p+k}}{2^{p(p+k)}} \\
%\nonumber  &= \lim_{p\rightarrow \infty} \sum_{h=0}^p \binom{p}{h} \left( \frac{p-1}{2} 2^{(p+k)/2}\right)^{p-h}\left(\frac{2^h + \frac{p+k-1}{2}}{2^p}\right)^{p+k} \\
%\label{ineq 003}  &\leq \lim_{p\rightarrow \infty} \left(\left( 1 + \frac{p+k-1}{2^{p+1}}\right)^{p+k} - \left(\frac{1 + \frac{p+k-1}{2^p}}{2}\right)^{p+k} \right.\\ 
%\nonumber &\ \ \ \ \ \ \ + \left. \sum_{h=0}^{p} \binom{p}{h} \left( \frac{p-1}{2} 2^{(p+k)/2}\right)^{p-h}\left(\frac{1 + \frac{p+k-1}{2^p}}{2}\right)^{p+k} \right)\\
%\nonumber  &= \lim_{p\rightarrow \infty} \left(\left( 1 + \frac{p+k-1}{2^{p+1}}\right)^{p+k}\right. \\
%\nonumber &\ \ \ \ \ \ \ +\left.\left(\frac{1 + \frac{p+k-1}{2^p}}{2}\right)^{p+k} \left( \left(\frac{p-1}{2} 2^{(p+k)/2} + 1\right)^{p}-1\right)\right) \\
%\nonumber  &= \lim_{p\rightarrow \infty} \left( 1 + \frac{p+k-1}{2^{p+1}}\right)^{p+k} \\
%\nonumber  &= 1,
%\end{align}
%with inequality \eqref{ineq 003} due to the observation that, if $h<p$, then 
%\begin{align*} \frac{2^h + \frac{p+k-1}{2}}{2^p}   &< \frac{1 + \frac{p+k-1}{2^p}}{2}. \end{align*}
\end{proof}

\section{The cycle form.}
\label{cycle form section}

%Classically, given a permutation group $G$ of degree $n$, the cycle index of $G$ is the polynomial 
%\begin{align*} Z(G) &= \frac{1}{\left| G\right|} \sum_{g\in G} \prod_{k=1}^n a_k^{c_k(g)}\end{align*}
%in polynomial indeterminates $a_1,a_2, \dots ,a_n$. Setting each of the indeterminates equal to $2$ yields the sum $\frac{1}{\left| G\right|} \sum_{g\in G} \prod_{k=1}^n 2^{c_k(g)}$.

We introduce a bit of notation:
\begin{definition}
Given $\beta\in \Sigma_q$ and a positive integer $r$, write $c_r(\beta)$ for the number of $r$-cycles in the expression of $\beta$ as a product of disjoint cycles. 
\end{definition}

\begin{definition}
Given permutation groups $G,H$, by the {\em cycle form} we mean the bilinear form $\langle -,-\rangle: \mathbb{Z}[G]\otimes_{\mathbb{Z}} \mathbb{Z}[H] \rightarrow \mathbb{Z}$ given by the formula
\begin{align*}
 \langle g,h \rangle &= \sum_{r,s\geq 1} \gcd(r,s)\cdot c_r(\alpha)\cdot c_s(\beta).
\end{align*}
\end{definition}
The case of immediate relevance is the case $G = \Sigma_p$ and $H = \Sigma_q$, so that the formula \eqref{harary formula} reduces to 
\begin{align} 
 \left|B_u(p,q)\right| &= \frac{1}{p!q!} \sum_{\alpha\in \Sigma_p}\sum_{\beta\in \Sigma_q} 2^{\langle\alpha,\beta\rangle}.
\end{align}

\begin{example}
Let $\beta\in \Sigma_q$. For the identity element $1\in\Sigma_p$, we have $\langle 1,\beta\rangle = p\cdot c(\beta)$. In particular, $\langle 1,1\rangle = pq$.
 If $\alpha\in\Sigma_p$ is an $\ell$-cycle with $\ell$ prime, then
\begin{align}\label{bracket for a cycle} \langle \alpha,\beta\rangle &= \langle 1,\beta\rangle - (\ell-1) \sum_{\ell \nmid s} c_s(\beta).
\end{align}
\end{example}

We develop a few basic properties of the cycle form. Whenever convenient, we will restrict to the case where the permutation groups are symmetric groups. 

We begin by showing that the cycle form is degenerate. Usually one wants bilinear forms to be nondegenerate, but the degenerateness of the cycle form is actually very useful, and in Lemma \ref{decomp eq lemma 2} we will see that it is key to the computability of the cycle form.
\begin{lemma}\label{decomp eq lemma}
Let $\alpha,\alpha^{\prime}$ be disjoint elements of a permutation group $G$. Then $(1 - \alpha)(1 - \alpha^{\prime})$ is in the radical of the cycle form. That is, for any $\beta$ in any permutation group $H$, we have
\begin{align}\label{decomp eq}
 \langle (1 - \alpha)(1 - \alpha^{\prime}),\beta\rangle &= 0.
\end{align}
\end{lemma}
\begin{proof}
Since $\alpha$ and $\alpha^{\prime}$ are disjoint, the number of fixed points of $\alpha\alpha^{\prime}$ is equal to the number of fixed points of $\alpha$ minus the number of {\em non}-fixed points of $\alpha^{\prime}$, i.e., 
\begin{align*}
 c_1(\alpha\alpha^{\prime}) &= c_1(\alpha) - (p-c_1(\alpha^{\prime})).
\end{align*}
Hence we have
\begin{align*}
 \langle \alpha \alpha^{\prime},\beta\rangle
  &= \sum_{r,s} \gcd(r,s) c_r(\alpha\alpha^{\prime}) c_s(\beta) \\
  &= \sum_{r>1} \sum_s \gcd(r,s) (c_r(\alpha) + c_r(\alpha^{\prime}))c_s(\beta)  + \sum_s (c_1(\alpha) + c_1(\alpha^{\prime}) - p)c_s(\beta) \\
  &= \langle \alpha,\beta\rangle + \langle \alpha^{\prime},\beta\rangle - \langle 1,\beta\rangle,
\end{align*}
and \eqref{decomp eq} follows.
\end{proof}

\begin{lemma}\label{decomp eq lemma 2}
Let $p,q$ be positive integers, and let $\alpha,\beta$ be elements of the symmetric groups $\Sigma_p$ and $\Sigma_q$, respectively. 
Then we have an equality
\begin{align}\label{decomp eq 2}
 \langle \alpha,\beta\rangle
  &= \sum_{j=1}^p c_j(\alpha)\langle \gamma_j, \beta\rangle + (1-c(\alpha))\cdot p\cdot c(\beta),
\end{align}
where $\gamma_j$ is any $j$-cycle in $\Sigma_p$.
\end{lemma}
\begin{proof}
Write $\alpha$ as a product of disjoint cycles, $\alpha = \alpha_1\alpha_2\dots \alpha_{c(\alpha)}$, including singleton cycles. Apply Lemma \ref{decomp eq lemma} repeatedly:
\begin{align*}
 \langle \alpha,\beta\rangle 
  &= \langle \alpha_1,\beta\rangle + 
  \langle \alpha_2\alpha_3\dots \alpha_{c(\alpha)},\beta\rangle - \langle 1,\beta\rangle \\
  &= \langle \alpha_1,\beta\rangle + 
     \langle \alpha_2,\beta\rangle + 
  \langle \alpha_3\dots \alpha_{c(\alpha)},\beta\rangle - 2\langle 1,\beta\rangle \\
  &= \dots \\
  &= \sum_{k=1}^{c(\alpha)} \langle \alpha_k,\beta\rangle - (c(\alpha)-1)\langle 1,\beta\rangle,\end{align*}
which is equal to the right-hand side of \eqref{decomp eq 2}.
\end{proof}

Our next task is to bound the value of $\langle 1 - \alpha,\beta\rangle$.
The bound $\langle 1 - \alpha,\beta\rangle \geq 0$ is straightforward to see, and when $\ell$ is prime, it is a trivial consequence of \eqref{bracket for a cycle}. 
A more interesting bound is obtained as follows. Suppose that $\ell$ is a positive integer. A quantity of basic interest is the difference $c(\beta) - \frac{q}{\ell}$, since it is positive if and only if the average length of a cycle in $\beta$, including length $1$ cycles, is greater than $\ell$. One of the fundamental properties of the cycle form is that it is bounded by the quantity $c(\beta) - \frac{q}{\ell}$, in the following sense:
\begin{lemma}\label{bound lemma 1a}
Let $\ell$ be a positive integer, and let $\gamma_{\ell} \in \Sigma_p$ be an $\ell$-cycle. Then, for all $\beta\in \Sigma_q$, we have
\begin{align}\label{bound 1a}
\langle 1 - \gamma_{\ell},\beta\rangle \geq (\ell-1) \left( c(\beta) - \frac{q}{\ell}\right).
\end{align}
\end{lemma}
\begin{proof}
Unpacking definitions, we have equalities
\begin{align}
\nonumber \left( (\ell - 1) \left( c(\beta) - \frac{q}{\ell}\right)\right) - \langle 1 - \gamma_{\ell} ,\beta\rangle 
  &=\left( \sum_s (p-\ell + \gcd(s,\ell)-p)c_s(\beta)\right)
   +\left( \sum_s (\ell-1)c_s(\beta)\right) - \frac{q(\ell-1)}{\ell} \\
\label{eq 3409f}  &=\left( \sum_s \left( \gcd(s,\ell) - 1\right)c_s(\beta)\right) - \frac{q(\ell-1)}{\ell}.
\end{align}
We claim that the sum $\sum_s \left( \gcd(s,\ell) - 1\right)c_s(\beta)$ is less than or equal to $\frac{q(\ell-1)}{\ell}$. The way to see this is to regard it as an optimization question: the numbers $c_1(\beta),c_2(\beta), \dots, c_q(\beta)$ must be nonnegative {\em integers} with the property that \begin{align*} \sum_j j\cdot c_j(\beta) &= q.\end{align*} Whichever choice of such integers $c_1(\beta),c_2(\beta), \dots, c_q(\beta)$ maximizes the sum\linebreak $\sum_s \left( \gcd(s,\ell) - 1\right)c_s(\beta)$, it can be no greater than the greatest possible value of the sum $\sum_s \left( \gcd(s,\ell) - 1\right)x_s$ where $x_1, \dots ,x_q$ are required only to be nonnegative {\em real numbers} satisfying $\sum_j j\cdot x_j = q$. It is elementary to verify that the choice of real numbers $x_1, \dots ,x_q$ satisfying those constraints, and maximizing the value of $\sum_s \left( \gcd(s,\ell) - 1\right)x_s$, is the choice given by letting $x_{\ell} = q/\ell$ and letting $x_s = 0$ for all $s\neq \ell$. In that case, the sum $\sum_s \left( \gcd(s,\ell) - 1\right)x_s$ is equal to $\frac{q(\ell-1)}{\ell}$. 

Hence $\sum_s \left( \gcd(s,\ell) - 1\right)c_s(\beta)$ can be no greater than $\frac{q(\ell-1)}{\ell}$, i.e., \eqref{eq 3409f} is negative. The bound \eqref{bound 1a} follows.
\end{proof}

\begin{lemma}\label{bound lemma 5}
Let $p$ be a positive integer, and let $\alpha\in\Sigma_p$. Then we have an inequality
\begin{align}
\label{ineq 0495}
 \sum_{j=1}^p \frac{c_j(\alpha)}{j} 
  &\geq \frac{c(\alpha) - p}{2}.
\end{align}
\end{lemma}
\begin{proof}
By elementary algebra, \eqref{ineq 0495} is equivalent to the inequality $\sum_{j=1}^p c_j(\alpha)(1-\frac{2}{j}) \leq p$, which is satisfied since $p\geq c(\alpha) = \sum_{j=1}^p c_j(\alpha)$. 
\end{proof}

\begin{theorem}\label{main thm}
Let $p,q$ be positive integers. 
Let $B_u(p,q)$ denote the set of unlabelled bicolored graphs with $p$ red vertices and $q$ blue vertices.
Then we have the inequality
\begin{align}
 \left| B_u(p,q)\right|
\label{ineq 1}  &\leq 2^{pq/2} ((\chi_{1/2}, \chi_{2^{q/2}})).
\end{align}
\end{theorem}
\begin{proof}
\begin{align}
 \sum_{\alpha\in\Sigma_p} 2^{\langle \alpha,\beta\rangle}
\label{eq 49901}  &= \sum_{\alpha\in\Sigma_p} 2^{(1-c(\alpha))\langle 1,\beta\rangle + \sum_{j=1}^p c_j(\alpha)\langle \gamma_j,\beta\rangle}\\
\nonumber  &= 2^{p\cdot c(\beta)}\sum_{\alpha\in\Sigma_p} 2^{-p\cdot c(\alpha)\cdot c(\beta) + \sum_{j=1}^p c_j(\alpha)\langle \gamma_j,\beta\rangle}\\
\label{eq 49902}  &\leq 2^{p\cdot c(\beta)}\sum_{\alpha\in\Sigma_p} 2^{-p\cdot c(\alpha)\cdot c(\beta) + \sum_{j=1}^p c_j(\alpha)\left( \langle 1,\beta\rangle -(j-1)(c(\beta) - q/j)\right)}\\
\nonumber  &= 2^{p\cdot c(\beta)}\sum_{\alpha\in\Sigma_p} 2^{(c(\beta)+q)\cdot c(\alpha)} 2^{-c(\beta)\sum_{j=1}^p j\cdot c_j(\alpha)} 2^{-q\sum_{j=1}^p c_j(\alpha)/j}\\ 
\label{eq 49903}  &\leq 2^{p\cdot c(\beta)}\sum_{\alpha\in\Sigma_p}2^{c(\alpha)\cdot (c(\beta) - q/2)}2^{p(q/2 - c(\beta))}\\
\nonumber  &= 2^{p\cdot c(\beta)}\sum_{\alpha\in\Sigma_p}2^{(c(\alpha)-p)(c(\beta)-q/2)} ,
\end{align}
with \eqref{eq 49901},\eqref{eq 49902}, and \eqref{eq 49903} due to Lemmas \ref{decomp eq lemma 2}, \ref{bound lemma 1a}, and \ref{bound lemma 5}, respectively.
Summing over $\beta$, we get
\begin{align}
 \sum_{\alpha\in\Sigma_p}\sum_{\beta\in\Sigma_q} 2^{\langle \alpha,\beta\rangle}
  &\leq 2^{pq/2} \sum_{\alpha\in\Sigma_p}\sum_{\beta\in\Sigma_q} 2^{c(\alpha)\cdot c(\beta)} (2^{-q/2})^{c(\alpha)} \\
  &= 2^{pq/2}\cdot p!\cdot q!\cdot (( \chi_{1/2},\chi_{2^{q/2}})),\end{align}
and \eqref{ineq 1} follows, using \eqref{harary formula}.
\end{proof}

\begin{corollary}\label{growth cor}
Let $k$ be a nonnegative integer. Then 
the number of bicolored graphs with $p$ red vertices and $p+k$ blue vertices grows asymptotically no faster than $\frac{2^{p(p+k)}}{p!(p+k)!}$.
That is, 
\begin{align}
 \lim_{p\rightarrow \infty} \frac{\left| B_u(p,p+k)\right|}{2^{p(p+k)}/(p!(p+k)!)} &\leq 1.
\end{align}
\end{corollary}
\begin{proof}
Immediate from Theorem \ref{main thm} and Proposition \ref{dir char product bound}. 
\end{proof}
We point out that the word ``nonnegative'' can safely be dropped from the statement of Corollary \ref{growth cor}, since $\left|B_u(p+k,p)\right| = \left|B_u(p,p+k)\right|$.

\begin{remark}\label{remark on atmaca-oruc}
In the paper \cite{MR3901841}, the following bounds are proven for $\left| B_u(p,q)\right|$:
\begin{equation}\label{atmaca-oruc bounds}
 \frac{\binom{p+2^q-1}{p}}{q!} 
  \leq \left| B_u(p,q)\right|
  \leq \frac{2\binom{p+2^q-1}{p}}{q!} .
\end{equation}
The authors of that paper also remark that ``an asymptotic formula is provided'' for $\left| B_u(p,q)\right|$ by the inequalities \eqref{atmaca-oruc bounds}. The factor of $2$ in \eqref{atmaca-oruc bounds} is a bit troubling: one does not know whether $\left| B_u(p,q)\right|$ grows asymptotically like $\frac{\binom{p+2^q-1}{p}}{q!}$ or like $\frac{2\binom{p+2^q-1}{p}}{q!}$. 

Of course the asymptotic growth rate of $\left| B_u(p,q)\right|$ depends on {\em how} $p$ and $q$ increase. Letting them increase at the same rate---i.e., letting $q = p+k$ and letting $p\rightarrow \infty$---this factor of $2$ in the asymptotic growth rate of 
is resolved by Corollary \ref{growth cor}: the left-hand side of \eqref{atmaca-oruc bounds}, $\frac{\binom{p+2^{p+k}-1}{p}}{(p+k)!}$, is the correct growth rate for $\left| B_u(p,p+k)\right|$. This is because 
\begin{align*}
\label{asympt 049} \lim_{p\rightarrow\infty} \frac{\binom{p+2^{p+k}-1}{p}/(p+k)!}{2^{p(p+k)}/(p!(p+k)!)}
  &= \lim_{p\rightarrow\infty} \frac{(p+2^{p+k}-1)!}{2^{p(p+k)}(2^{p+k}-1)!} \\
  &= \lim_{p\rightarrow\infty} \frac{(2^{p+k})^{\overline{p}}}{(2^{p+k})^p} \\ &= 1.
\end{align*}

We hope the reader will forgive us for presenting a bit of numerics to give a sense of how our bound for $\left| B_u(p,p+k)\right|$, proven in Theorem \ref{main thm}, compares to Atmaca and Oru\c{c}'s upper bound $\frac{2\binom{p+2^{p+k}-1}{p}}{(p+k)!}$ proven in \cite{MR3901841}. Here is a table: 
\begin{center}
\begin{tabular}{c|ccccc}
 & k=0 & k=1 & k=2 & k=3 & k=4 \\
\hline
p= 3 & 0.67853 & 0.448352 & 0.281421 & 0.164794 & 0.089167 \\
p= 6 & 0.236554 & 0.278629 & 0.321008 & 0.355492 & 0.37623 \\
p= 9 & 0.401765 & 0.581412 & 0.769003 & 0.943255 & 1.089729 \\
p= 12 & 0.737444 & 0.964918 & 1.174011 & 1.352241 & 1.495579 \\
p= 15 & 1.13395 & 1.332052 & 1.495158 & 1.62365 & 1.721639 \\
p= 18 & 1.488057 & 1.620956 & 1.722684 & 1.798768 & 1.854731 \\
p= 21 & 1.731173 & 1.805571 & 1.860243 & 1.899968 & 1.928601 \\
p= 24 & 1.869913 & 1.907043 & 1.933771 & 1.95291 & 1.966564 \\
p= 27 & 1.940359 & 1.957629 & 1.969938 & 1.978691 & 1.984905 \\
p= 30 & 1.973633 & 1.981317 & 1.98677 & 1.990635 & 1.993373 \\
p= 33 & 1.98864 & 1.99196 & 1.994311 & 1.995976 & 1.997154 \\
p= 36 & 1.995199 & 1.996604 & 1.997598 & 1.998301 & 1.998799 \\
p= 39 & 1.998002 & 1.998587 & 1.999001 & 1.999293 & 1.9995 \\
p= 42 & 1.999179 & 1.999419 & 1.999589 & 1.99971 & 1.999795 \\
p= 45 & 1.999666 & 1.999764 & 1.999833 & 1.999882 & 1.999917 \\
p= 48 & 1.999866 & 1.999905 & 1.999933 & 1.999952 & 1.999966 
\end{tabular}
\end{center}
The entry marked with row number $p$ and column number $k$ in this table is the ratio of our upper bound for $\left| B_u(p,p+k)\right|$ from Theorem \ref{main thm} to Atmaca--Oru\c{c}'s upper bound. The ratios are all rounded to the first six decimal points. It is evident that the Atmaca--Oru\c{c} upper bound for $\left| B_u(p,p+k)\right|$ beats our upper bound for small values of $p$, but ours quickly overtakes the Atmaca--Oru\c{c} upper bound and converges to half of their upper bound.
\end{remark}

\begin{theorem}\label{main cor}
Let $f(p,q)$ be the proportion of elements in the power set $P(\{1, \dots ,p\} \times \{1, \dots q\})$ which are members of free $\Sigma_p\times\Sigma_q$-orbits. 
Let $k$ be an integer. Then the limit $\lim_{p\rightarrow \infty} f(p,p+k)$ is equal to $1$.
\end{theorem}
\begin{proof}
Consider the sizes of the orbits of the action of $\Sigma_p\times\Sigma_{q}$ on $P(\{1, \dots ,p\}\times \{1, \dots ,q\})$. The size of each orbit is a divisor of $p!q!$. There are $\left|B_u(p,q)\right|$ orbits, and the sum of their sizes is $2^{pq}$. If there are no free orbits (i.e., no orbits of size $p!q!$), then we must have $\frac{p!q!}{2} \cdot \left|B_u(p,q)\right| \geq 2^{pq}$. By a similar argument, if
\begin{align*}
 \left(\left| B_u(p,q)\right| - r\right)\frac{p!q!}{2} 
  &< 2^{pq} - p!q!r,
\end{align*}
then $P(\{1, \dots ,p\}\times \{1, \dots ,q\})$ must have more than $r$ free orbits.

Hence, in the case $q=p+k$, we have the inequality
\begin{align*}
 f(p,p+k) &\geq % \frac{2^{1+p(p+k)} - p!(p+k)!\left| B_u(p,p+k)\right|}{2^{p(p+k)}} \\ &= 
 2 - \frac{p!(p+k)!\left| B_u(p,p+k)\right|}{2^{p(p+k)}}.
\end{align*}
In the limit, due to Corollary \ref{growth cor}, we have
\begin{align*}
 \lim_{p\rightarrow\infty}f(p,p+k) &\geq 1.
\end{align*}
Since $f(p,p+k)$ is a ratio, it cannot be less than $1$. Hence the limit must be $1$, as claimed.
\end{proof}

\def\cprime{$'$} \def\cprime{$'$} \def\cprime{$'$} \def\cprime{$'$}

%\bibliography{/home/asalch/texmf/tex/salch.bib}{}
%\bibliographystyle{plain}
\end{document}